\documentclass[12pt, reqno]{amsart}
\usepackage{amsmath, amsthm, amscd, amsfonts, amssymb, graphicx, color}
\usepackage[ plainpages]{hyperref}

\newtheorem{theorem}{Theorem}[section]
\newtheorem{lemma}[theorem]{Lemma}
\newtheorem{proposition}[theorem]{Proposition}
\newtheorem{corollary}[theorem]{Corollary}

\theoremstyle{definition}

\theoremstyle{remark}
\newtheorem{remark}[theorem]{Remark}

\newcommand{\tr}{{\rm{tr}}}

\usepackage[all]{xy}
\numberwithin{equation}{section}

\newtheorem*{theorem*}{Theorem}

\begin{document}
\title[On the matrix harmonic mean ]{ On the matrix harmonic mean}\author[M. Sababheh]{M. Sababheh}\address{Department of Basic Sciences, Princess Sumaya University For Technology, Al Jubaiha, Amman 11941, Jordan.}\email{\textcolor[rgb]{0.00,0.00,0.84}{sababheh@psut.edu.jo, sababheh@yahoo.com}}


\maketitle
\begin{abstract}
The main goal of this article is to present new types of  inequalities refining and reversing inequalities of the harmonic mean of scalars and matrices. Furthermore, implementing the spectral decomposition of positive matrices, we present a new type of inequalities treating certain harmonic matrix perturbation.
\end{abstract}
\section{introduction and Motivation}
For two positive numbers $a,b$ and $0\leq t\leq 1,$ the Heinz means are defined by
$$H_t(a,b)=\frac{a^{1-t}b^{t}+a^{t}b^{1-t}}{2}.$$ These means attracted researchers working in the field of matrix inequalities, where the matrix versions of these means have their role. In the setting of matrices, $\mathbb{M}_n$ will denote the algebra of all complex $n\times n$ matrices, $\mathbb{M}_n^{+}$ will denote the cone of positive semi-definite matrices in $\mathbb{M}_n$ while $\mathbb{M}_n^{++}$ will denote the cone of strictly positive matrices in $\mathbb{M}_n$. A possible matrix version of the Heinz means is $$\frac{A^{1-t}B^{t}+A^{t}B^{1-t}}{2}, A,B\in\mathbb{M}_n^{+}.$$ Using the notation $\||\;\;\||$  for an arbitrary unitarily invariant norm on $\mathbb{M}_n$, the quantity
$$\frac{1}{2}\||A^{1-t}XB^{t}+A^{t}XB^{1-t}\||, A,B\in\mathbb{M}_n^+, X\in\mathbb{M}_n$$ happens to be among the most natural possible matrix versions of the numerical Heinz means.

Numerous researchers have investigated these means and their inequalities. For $0\leq t\leq 1, A,B\in\mathbb{M}_n^+$ and $X\in\mathbb{M}_n$, the inequality
$$2\||A^{\frac{1}{2}}XB^{\frac{1}{2}}\||\leq \||A^{1-t}XB^{t}+A^{t}XB^{1-t}\||\leq \||AX+XB\||$$
is well known by the Heinz inequality \cite{B2}. Refining, reversing and obtaining variants of this inequality received a considerable attention in the literature that the interested reader can see in \cite{bak,conde,kitt_convex,kittanehmanasreh,krnic,saboam,sab_conv,zhao}, for example.

Going back to the numerical version of the Heinz means $H_t(a,b)$, we have the known inequality
\begin{equation}\label{heinz_original_scalars}
\sqrt{ab}\leq H_t(a,b)\leq \frac{a+b}{2}, 0\leq t\leq 1.
\end{equation}
This inequality is usually interpreted by saying that the Heinz means interpolate between the geometric  and arithmetic means. In fact, the Heinz means can be thought of an average of certain geometric means. Recall that for $a,b>0, 0\leq t\leq 1,$ the weighted geometric mean is defined by $a\#_t b=a^{1-t}b^{t}.$ Therefore, $H_t(a,b)=\frac{a\#_{t}b+a\#_{1-t}b}{2}$ and \eqref{heinz_original_scalars} can be simply written as
$$a\# b\leq H_t(a,b)\leq a\nabla b,$$ where $a\#b=a\#_{\frac{1}{2}}b$ is the geometric mean and $a\nabla b=a\nabla_{\frac{1}{2}}b$ is the arithmetic mean computed via the formula $a\nabla_t b=(1-t)a+t b.$

Among the most interesting inequalities of the Heinz means is its comparison with the Heron means, defined for $a,b>0$ and $0\leq t\leq 1$ by
$$K_t(a,b)=(1-t)(a\#b)+t(a\nabla b).$$ This comparison was shown in \cite{bhatiaint} as follows
\begin{equation}\label{heron_heinz_bhatia}
H_t(a,b)\leq K_{\alpha(t)}(a,b)\;{\text{where}}\;\alpha(t)=(1-2t)^2.
\end{equation}
Then this inequality was further explored in \cite{singh}, for example.

Our motivation of the current work begins with \eqref{heron_heinz_bhatia}. It is a main goal of this paper to  prove versions of \eqref{heron_heinz_bhatia} for the harmonic mean. Recall that for $a,b>0$ and $0\leq t\leq 1$, the weighted harmonic mean is defined by $a!_t b=((1-t)a^{-1}+t b^{-1})^{-1}.$ In particular $a!_{\frac{1}{2}}b$ will be simply denoted by $a!b.$

For two positive numbers $a,b$ and $0\leq t\leq 1,$ we define the harmonic Heinz means  by
$$!_t(a,b)=\frac{a!_{t}b+a!_{1-t}b}{2}.$$ Since $a!_{t}b\leq a\#_t b\leq a\nabla_t b,$ \cite{liao}, we have $!_t(a,b)\leq H_t(a,b)\leq \frac{a+b}{2},$ for $0\leq t\leq 1.$ Also, direct computations show that $a!b\leq !_t(a,b).$

In this paper, we also define the Heron-harmonic means by
$$F_t(a,b)=(1-t)(a!b)+t(a\nabla b), a,b>0$$ similar to the definition of $K_t(a,b)$ above.

We shall prove that
$$!_t(a,b)\leq F_{\alpha(t)}(a,b), \alpha(t)=(1-2t)^2,$$ an inequality similar to \eqref{heinz_original_scalars}. However, the method of proof is a result of a new approach to tickle these inequalities.

This new approach is  dealing with the geometric meaning of these inequalities. More precisely, we will see that the quadratic polynomial defined by $F_{\alpha(t)}(a,b)$ is the quadratic polynomial interpolating $!_t(a,b)$ at $t=0,\frac{1}{2},1.$ Then the comparison $!_t(a,b)\leq F_{\alpha(t)}(a,b)$ comes as a special case of a more general statement.\\

Notice that the inequality $!_t(a,b)\leq F_{\alpha(t)}(a,b)$ can be simply written as
\begin{equation}\label{heinz_har_intro}
 !_t(a,b)+4t(1-t)(a\nabla b-a!b)\leq a\nabla b,
\end{equation}
 which is a refinement of the inequality $!_t(a,b)\leq a\nabla b.$ However, this inequality is an interesting refinement as it presents a quadratic refining term in $t$. This last inequality is similar to the recent refinement of the Heinz inequality $$H_t(a,b)+4t(1-t)(a\nabla b-a\#b)\leq a\nabla b,$$ shown in \cite{krnic}. This inequality of \cite{krnic} follows from \eqref{heron_heinz_bhatia}.

Observe that the refining term in \eqref{heinz_har_intro} is $4t(1-t)(a\nabla b-a!b)$, which is added to $!_t(a,b).$ Once we establish these inequalities, we present multiplicative versions; where the refining term is multiplied by $!_t(a,b).$

Then following the same guideline, we prove quadratic refinements and reverses of the arithmetic-harmonic and geometric-harmonic inequalities. The geometric meaning of these refinements will be similar to that of the Harmonic-Heinz means, and will not be emphasized further.

We emphasize that the technique we use to prove these inequalities is new and is different from the techniques used in all references treating this topic. However, this idea is motivated by our recent work in \cite{kitt_mos_sab}, where the Heinz means themselves were explored.

Once the above numerical inequalities are proved, we prove their matrix versions. For example, we show that
$$!_t(A,B)+\frac{t(1-t)}{\tau(1-\tau)}\left[A\nabla B-!_{\tau}(A,B)\right]\leq A\nabla B,$$
$$\det(A\nabla B)^{\frac{1}{n}}\geq \det(!_t(A,B))^{\frac{1}{n}}+\frac{t(1-t)}{\tau(1-\tau)}\det(A\nabla B-!_{\tau}(A,B))^{\frac{1}{n}},$$
$$\left(\frac{AB^{-1}+BA^{-1}+2I}{4}\right)^{4t(1-t)}!_t(A,B)\leq A\nabla B,$$
for certain parameters and $A,B\in\mathbb{M}_n^{++}.$ More results treating the matrix arithmetic-harmonic and geometric-harmonic means will be presented too.

Then, we introduce a new type of inequalities treating the matrix harmonic means. For example, we show that
$$\|X\|_2^2\;\|(1-t)A^{-1}X+t XB^{-1}\|_2^{-1}\leq\|(1-t)AX+t XB\|_2$$ for $A,B\in\mathbb{M}_n^{++}, X\in\mathbb{M}_n$ and $0\leq t\leq 1.$ See Theorem \ref{thm_har_young} for the details. The idea of this inequality is new and it extends similar matrix versions of the geometric mean to the context of harmonic mean. The proof of this new inequality is based on convex functions.

For the used notation, if $0\leq \tau\leq 1,$ we will use the functions
$$r(\tau):=\min\{\tau,1-\tau\}\;{\text{and}}\;R(\tau)=\max\{\tau,1-\tau\}.$$
\section{Main results}
In this part of the paper, we present some new scalar inequalities for the harmonic mean $a!_tb$ and the harmonic-Heinz mean $!_t(a,b).$ These inequalities will be needed later to prove the corresponding matrix versions.

\subsection{The Harmonic-Heinz means}
\begin{proposition}
For $c>0$, let $$f(t)=\frac{1+c-(1!_{t}c+1!_{1-t}c)}{t(1-t)}.$$ Then $f$ is decreasing on $\left(0,\frac{1}{2}\right)$ and is increasing on
$\left(\frac{1}{2},1\right).$
\end{proposition}
\begin{proof}
Notice first that if $c=1$, then $f=0$ and there is nothing to prove. So, without loss of generality, $c\not=1.$ Observe that $f(t)=g(t)+g(1-t),$ where
$$g(t)=\frac{1\nabla_t c-1!_{t}c}{t(1-t)}.$$ Direct computations show that
$$g''(t)=\frac{2(c-1)^4}{((1-t)c+t)^3}.$$ Since $0<t<1$ and $c>0$, it follows that $g''>0$ and $g$ is convex on $(0,1).$ Since $g$ is convex and $f(t)=g(t)+g(1-t),$ it follows that $f$ is convex on $(0,1).$ But then, either $f$ is monotone on $(0,1)$ or $f$ is decreasing on $(0,t_0)$ and increasing on $(t_0,1)$ for some $t_0\in(0,1).$ We assert that $t_0=\frac{1}{2}.$ Notice first that $f'\left(\frac{1}{2}\right)=0.$ Therefore, if $t\in (0,1)$ then, for some $\xi_t$ between $t$ and $\frac{1}{2}$,
$$f(t)=f\left(\frac{1}{2}\right)+\frac{f''(\xi_t)}{2}\left(t-\frac{1}{2}\right)^2.$$ Since $f$ is convex, it follows that $f''\geq 0$ and hence $f(t)\geq f\left(\frac{1}{2}\right).$ This proves that $t_0=\frac{1}{2}$.
\end{proof}
In particular, the function $f(t)=\frac{1+c-(1!_{t}c+1!_{1-t}c)}{t(1-t)}$ attains its minimum at $t_0=\frac{1}{2}$ and $f(t)\geq f\left(\frac{1}{2}\right).$
This entails the following arithmetic-harmonic Heinz-type inequality.
\begin{corollary}\label{heron_harm_cor}
Let $a,b>0$ and $0\leq t\leq 1.$ Then
$$\frac{a!_t b+a!_{1-t}b}{2}+4t(1-t)(a\nabla b-a!b)\leq a\nabla b.$$
\end{corollary}

Notice that Corollary \ref{heron_harm_cor}  can be read simply
as $$!_t(a,b)\leq F_{\alpha(t)}(a,b)\;{\text{where}}\;\alpha(t)=1-4t(1-t).$$

A full description of the interpolation of the Heinz harmonic mean is as follows.
 \begin{corollary}\label{full_comp_heinz}
  Let $a,b>0$ and fix $\tau\in (0,1).$ Then
 $$\frac{a\nabla b-!_{\tau}(a,b)}{\tau(1-\tau)}\leq \frac{a\nabla b-!_{t}(a,b)}{t(1-t)}$$ for $ t\leq r(\tau)$ or
 $t\geq R(\tau).$ On the other hand, if $r(\tau)\leq t\leq R(\tau)$, the inequality is reversed.
 \end{corollary}
Observe that the above inequality can be written as
\begin{equation}\label{harmonic_less}
!_{t}(a,b)\leq a\nabla b+\frac{t(1-t)}{\tau(1-\tau)}\left[!_{\tau}(a,b)-a\nabla b\right],
\end{equation}
when $t\leq r(\tau)$ or $t\geq R(\tau)$, while the inequality is reversed when $r(\tau)\leq t\leq R(\tau).$ Notice that the right hand side of
\eqref{harmonic_less} is a quadratic polynomial in $t$, which coincides with $!_t$ at $t=0,\tau,1.$ Therefore,
$$Q_{\tau}(t;a,b):=a\nabla b+\frac{t(1-t)}{\tau(1-\tau)}\left[!_{\tau}(a,b)-a\nabla b\right]$$ is the quadratic polynomial interpolating $!_t(a,b)$ at $t=0,\tau,1.$

Adopting the notation $Q_{\tau}$ as above, we have $!_{t}(a,b)\leq Q_{\tau}(t;a,b)$ when $t\leq r(\tau)$ or $t\geq R(\tau),$ while the inequality is reversed when $r(\tau)\leq t\leq R(\tau).$ This provides a geometric meaning of these refinements.



Our next target is to present multiplicative refinements and reverses.
\begin{proposition}
For $c>0$, let
$$f(t)=\left(\frac{1+c}{1!_t c+1!_{1-t}c}\right)^{\frac{1}{t(1-t)}}.$$ Then $f$ is decreasing on $\left(0,\frac{1}{2}\right)$ and is increasing on
$\left(\frac{1}{2},1\right).$
\end{proposition}
\begin{proof}
Consider the function
$$g(t)=\log f(t)=\frac{\log(1+c)-\log(1!_t c+1!_{1-t}c)}{t(1-t)}.$$ Then
$$g'(t)=\frac{(2t-1)h(t)}{t^2(1-t)^2},$$ where
$$h(t)=\frac{(c -1)^2 (t-1) t}{(1-t+c\;t) (c + t - c\; t)} +
 \log(1 + c) - \log\left( \frac{c}{1-t+c\;t} + \frac{c}{c + t - c\; t}\right).$$ Moreover,
 $$h'(t)=\frac{(c-1)^4t(1-t)(1-2t)}{(1-t+c\;t)^2(c-c\;t+t)^2}.$$ Clearly, $h'(t)>0$ when $t<\frac{1}{2}$ and $h'(t)<0$ when $t>\frac{1}{2}.$ Since $h(0)=h(1)=0$, it follows that $h(t)\geq 0$. But then, $g'(t)<0$ when $t<\frac{1}{2}$ and $g'(t)>0$ when $t>\frac{1}{2}.$ This completes the proof.
\end{proof}
In particular, the function $f(t)=\left(\frac{1+c}{1!_t c+1!_{1-t}c}\right)^{\frac{1}{t(1-t)}}$ attains its minimum at $t_0=\frac{1}{2}.$ This entails the following multiplicative refinement of the Heinz-type inequality $!_t(a,b)\leq a\nabla b.$
\begin{corollary}\label{cor_1/2_multi}
 Let $a,b>0$ and $0<t<1.$ Then
\begin{equation}\label{multi_ref_heinz_har}
\left(\frac{a\nabla b}{a!b}\right)^{4t(1-t)}!_t(a,b)\leq a\nabla b.
\end{equation}
\end{corollary}
Notice that \eqref{multi_ref_heinz_har} is a refinement of  $!_t(a,b)\leq a\nabla b$ because $\frac{a\nabla b}{a!b}\geq 1.$ Furthermore, the constant $\left(\frac{a\nabla b}{a!b}\right)$ is known as the Kantorovich constant, and has appeared in many recent refinements of some mean inequalities. The reader is referred to \cite{zuo} as a sample of these studies.\\
A full multiplicative comparison is then given as follows.
\begin{corollary}\label{multi_har_cor}
Let $a,b>0$ and $0<t,\tau<1.$ If $t\leq r(\tau)$ or $t\geq R(\tau)$, then
$$!_{t}(a,b)\left(\frac{a\nabla b}{!_{\tau}(a,b)}\right)^{\frac{t(1-t)}{\tau(1-\tau)}}\leq a\nabla b.$$ On the other hand, if $r(\tau)\leq t\leq R(\tau)$, the inequality is reversed.
\end{corollary}

Letting $a=1$ in Corollaries \ref{full_comp_heinz} and \ref{multi_har_cor}, we obtain the inequalities
\begin{equation}\label{needed_1_further}
\frac{1\nabla b-!_{\tau}(1,b)}{1\nabla b-!_t(1,b)}\leq (\geq) \frac{\tau(1-\tau)}{t(1-t)}
\end{equation}
 and
\begin{equation}\label{needed_2_further}
\frac{\log 1\nabla b-\log!_{\tau}(1,b)}{\log 1\nabla b-\log !_t(1,b)}\leq (\geq) \frac{\tau(1-\tau)}{t(1-t)}.
\end{equation}
Inequalities \eqref{needed_1_further} and \eqref{needed_2_further} motivate the question about the relation between the quotients
$$\frac{1\nabla b-!_{\tau}(1,b)}{1\nabla b-!_t(1,b)}\;{\text{and}}\;\frac{\log 1\nabla b-\log!_{\tau}(1,b)}{\log 1\nabla b-\log !_t(1,b)}.$$

First of all, notice that the function $f(t)=!_t(1,b)$ is decreasing on $\left[0,\frac{1}{2}\right],$ increasing on $\left[\frac{1}{2},1\right]$ and is symmetric about $t=\frac{1}{2}.$ Therefore, if $\tau,t\in [0,1]$, then
\begin{equation}\label{needed_3_further}
!_t(1,b)\leq !_{\tau}(1,b),\;{\text{when}}\;r(\tau)\leq t\leq R(\tau),
\end{equation}
while we have the reversed inequality if $t\leq r(\tau)$ or $t\geq R(\tau).$
\begin{proposition}
Let $a,b>0$ and $0\leq \tau,t\leq 1.$ If $t\leq r(\tau)$ or $t\geq R(\tau),$ then
$$\frac{a\nabla b-!_{\tau}(a,b)}{a\nabla b-!_t(a,b)}\leq \frac{\log a\nabla b-\log!_{\tau}(a,b)}{\log a\nabla b-\log !_t(a,b)}.$$ On the other hand, if $r(\tau)\leq t\leq R(\tau),$ the inequality is reversed.
\end{proposition}
\begin{proof}
Without loss of generality, we may assume that $a=1.$ We begin with the case $r(\tau)\leq t\leq R(\tau).$ In this case, we have $!_t(1,b)\leq !_{\tau}(1,b)$ from \eqref{needed_3_further}. Dividing by $1\nabla b$ implies $\frac{!_t(1,b)}{1\nabla b}\leq \frac{!_{\tau}(1,b)}{1\nabla b}.$ Since the function $g(x)=\frac{\ln x}{x-1}$ is decreasing on $(0,\infty),$ it follows that
$$g\left(\frac{!_{\tau}(1,b)}{1\nabla b}\right)\leq g\left(\frac{!_t(1,b)}{1\nabla b}\right).$$ Simplifying this inequality gives the desired inequality in the case $r(\tau)\leq t\leq R(\tau).$ On the other hand, if $t\leq r(\tau)$ or $t\geq R(\tau),$ then $\frac{!_t(1,b)}{1\nabla b}\geq \frac{!_{\tau}(1,b)}{1\nabla b}.$ In this case,
$$g\left(\frac{!_t(1,b)}{1\nabla b}\right)\leq g\left(\frac{!_{\tau}(1,b)}{1\nabla b}\right),$$ which implies the desired inequality when $t\leq r(\tau)$ or $t\geq R(\tau).$
\end{proof}

\subsection{The arithmetic-harmonic mean inequality}
In this part of the paper, we present quadratic refinements and reverses of the arithmetic-harmonic mean inequality $a!_tb\leq a\nabla_t b, 0\leq t\leq 1.$\\
The proof of the following result is an immediate calculus application, where one step computation shows that the function $f(t)=\frac{1\nabla_t c-1!_t c}{t(1-t)}$ is increasing when $c>1$ and is decreasing when $c<1.$
\begin{proposition}\label{arith_har_num_first}
If $(b-a)(\tau-\nu)\geq0$, then
\begin{eqnarray*}
\tau(1-\tau)(a\nabla_{\nu}b-a!_{\nu}b)\leq\nu(1-\nu)(a\nabla_{\tau}b-a!_{\tau}b).
\end{eqnarray*}
On the other hand, if $(b-a)(\tau-\nu)\leq0$ then
\begin{eqnarray*}
\tau(1-\tau)(a\nabla_{\nu}b-a!_{\nu}b)\geq\nu(1-\nu)(a\nabla_{\tau}b-a!_{\tau}b).
\end{eqnarray*}
\end{proposition}
This is a quadratic refinement (reverse) of the arithmetic-harmonic mean inequality $a!_{t}b\leq a\nabla_tb.$ On the other hand, multiplicative versions can be proved as follows.

\begin{lemma}\label{arith_har_mono}
For $c>0$, define $f:(0,1)\to [0,\infty)$ by $$f(t)=\left(\frac{1\nabla_{t}c}{1!_{t}c}\right)^{\frac{1}{t(1-t)}}.$$ Then $f$ is decreasing on $\left(0,\frac{1}{2}\right)$ and is increasing on $\left(\frac{1}{2},1\right).$
\end{lemma}
\begin{proof}
Let $F(t)=\log f(t)$. That is $$F(t)=\frac{\log\left((1-t+t c)(1-t+t c^{-1})\right)}{t(1-t)}.$$ Then
$$F'(t)=\frac{2t-1}{t^2(1-t)^2}g(c),$$ where
$$g(c)=\frac{(c-1)^2(t-1)t}{(1-t+t c)(c+t-t c)}+\log\left((1-t+t c^{-1})(1-t+t c)\right).$$ Now
$$g'(c)=\frac{(c-1)^3(c+1)(1-t)^2t^2}{c(1-t+t c)^2(c+t-t c)^2}.$$
If $c<1$ then $g'(c)<0$ and $g$ is decreasing. That is $g(c)\geq g(1)=0$ when $c<1$. But since $F'(t)=\frac{2t-1}{t^2(1-t)^2}g(c)$ it follows that $F'<0$ when $t<\frac{1}{2}$ and $F'>0$ when $t>\frac{1}{2}.$ This proves the case $c<1.$ A similar argument implies the case $c>1.$
\end{proof}
For the following result, we use the notation $K(a,b)$ to denote the Kantorovich constant defined by $K(a,b)=\frac{a\nabla b}{a!b}=\left(\frac{a\nabla b}{a\#b}\right)^2.$ This constant has bees used recently as a refining factor in these inequalities. See \cite{liao,zuo} for example.
\begin{corollary}\label{arith_har_kanto_cor}
Let $a,b>0$. Then for $0<t<1,$
$$(a!_{t}b)K(a,b)^{4t(1-t)}\leq a\nabla_{t}b.$$
\end{corollary}
\begin{proof}
Let $c=\frac{b}{a}$ in $$f(t)=\left(\frac{1\nabla_{t}c}{1!_{t}c}\right)^{\frac{1}{t(1-t)}}.$$ By Lemma \ref{arith_har_mono}, $f$ attains its minimum at $t=\frac{1}{2}.$ That is $f(t)\geq f(1/2).$ This proves the desired inequality.
\end{proof}
The
 above inequality refines the known inequality \cite{liao} $$(a!_{t}b)K(a,b)^{2r}\leq a\nabla_{t}b, r=\min\{t,1-t\}$$ because $K(a,b)\geq 1$ and $4t(1-t)> 2\min\{t,1-t\}$ when $0<t<1.$

On the other hand, square versions  are as follows.
\begin{lemma}
For $c>1$, define $f:(0,1)\to[0,\infty)$ by $$f(t)=\frac{(1\nabla_{t}c)^2-(1!_{t}c)^2}{t(1-t)}.$$
\begin{enumerate}
\item If $c<1$, then $f$ is decreasing on $(0,1)$ and
\item if $c>1$, then $f$ is increasing on $(0,1)$.
\end{enumerate}
\end{lemma}
\begin{proof}
For the given $f$, $$f'(t)=\frac{(1-c)^3}{((1-t)c+t)^3}g(t),\;{\text{where}}\;
g(t)=-c (3 + c) + (-1 + c^2) t.$$ Then $g'(t)=c^2-1$, which is negative when $c<1$. Thus, if $c<1$, $g(t)\leq g(0)<0$ and hence $f'(t)<0$. That is, $f$ is decreasing when $c<1$. A similar argument implies the other case.
\end{proof}
\begin{corollary}\label{arith_har_square}
Let $a,b>0$ and $0<\nu,\tau<1$. If $(\tau-\nu)(b-a)>0$, then
$$\frac{(a\nabla_{\nu}b)^2-(a!_{\nu}b)^2}{\nu(1-\nu)}\leq \frac{(a\nabla_{\tau}b)^2-(a!_{\tau}b)^2}{\tau(1-\tau)}.$$ The inequality is reversed if $(\tau-\nu)(b-a)<0.$
\end{corollary}

\begin{remark}
Having introduced our numerical quadratic refinements and reverses, we compare these results with the linear inequalities. We have seen that, for $a,b>0$ and $0\leq t\leq 1,$ one has \cite{liao}
\begin{equation}\label{liao_ineq}
2r(t)(a\nabla b-a!b)\leq a\nabla_t b-a!_t b\leq 2R(t)(a\nabla b-a!b),
\end{equation}
where $r(t)=\min\{t,1-t\}$ and $R(t)=\max\{t,1-t\}.$
On the other hand, under certain ordering conditions, we have the quadratic refinement or reverse
$$a\nabla_t b-a!_t b\leq (\geq) 4t(1-t)(a\nabla b-a!b).$$ It is natural to ask about the advantage of introducing a quadratic refinement or reverse over the linear ones.\\
Direct calculations show that, for $0\leq t\leq 1,$ one has $r(t)\leq 2t(1-t)$ and $R(t)\geq 2t(1-t).$ Therefore, when $(b-a)(2t-1)\geq 0,$ we have
\begin{align*}
a!_tb+2r(t)(a\nabla b-a!b)&\leq a!_tb+4t(1-t)(a\nabla b-a!b)\leq a\nabla_t b,
\end{align*}
which is a refinement of the first inequality in \eqref{liao_ineq}. On the other hand, if $(b-a)(2t-1)\leq 0,$ we have
\begin{align*}
a!_tb+2R(t)(a\nabla b-a!b)&\geq a!_tb+4t(1-t)(a\nabla b-a!b) \geq a\nabla_t b,
\end{align*}
which is a refinement of the second inequality in \eqref{liao_ineq}. Therefore, introducing quadratic refinements serves as introducing one-term refinements of the already existing linear refinements.\\
A similar argument applies for the multiplicative versions.
\end{remark}

\subsection{The geometric-harmonic mean inequality}
We conclude the numerical versions by presenting some refinements and reverses of the geometric-harmonic mean inequality $a!_tb\leq a\#_tb, 0\leq t\leq 1.$
\begin{lemma}\label{geo_har_additive_mono}
For $c>1$, define $f:(0,1)\to [0,\infty)$ by $$f(t)=\frac{1\#_{t}c-1!_{t}c}{t(1-t)}.$$
Then $f$ is increasing on $(0,1).$
\end{lemma}
\begin{proof}
Notice first that $f(t)=F(t)G(t)$ where $$F(t)=\frac{(1-t)c^{t+1}+tc^{t}-c}{t(1-t)}\;{\text{and}}\;G(t)=\frac{1}{(1-t)c+t}.$$ Clearly, $G$ is decreasing if $c<1$ and is increasing if $c>1.$ The main calculations are for $F$. Notice that
$$F'(t)=\frac{g(c)}{t^2(1-t)^2},$$ where
$$g(c)=c - 2 c\; t + c^t (-c (-1 + t)^2 + t^2 + (c (-1 + t) - t) (-1 + t) t \log c).$$ Then,
$$g'(c)=\frac{(c - 2 c t +
 c^t (-c (-1 + t)^2 + t^2 + (-1 + t) t (-c + (-1 + c) t^2) \log c))}{c}$$ and
$$g''(c)=c^{-2 + t} (-1 + t)^2 t^2 h(t)\;{\text{where}}\; h(t)=-1 + c + (c + (-1 + c) t) \log c.$$ Moreover,
$h'(t)=(c-1)\log c>0.$ Since $h'(t)\geq 0, h(t)\geq h(0)= -1 + c + c \log c\geq 0 $ when $c>1$. Consequently,  $ g''(c)\geq 0, g'(c)\geq g'(1)=0 $ and $g(c)\geq g(1)=0.$ Therefore, $F'(t)\geq 0 $  and  $F$ is increasing.

 This completes the proof.
\end{proof}

Now Lemma \ref{geo_har_additive_mono} entails the following comparison between the geometric and harmonic means.
\begin{corollary}\label{geo_har_add}
Let $a,b>0$ and $0<\nu,\tau<1.$ If $(\tau-\nu)(b-a)>0$ then
$$\frac{a\#_{\nu}b-a!_{\nu}b}{\nu(1-\nu)}\leq \frac{a\#_{\tau}b-a!_{\tau}b}{\tau(1-\tau)}.$$
\end{corollary}

\begin{lemma}
For $c>0,$ define f$:(0,1)\to [0,\infty)$ by
$$f(t)=\left(\frac{1\#_{t}c}{1!_{t}c}\right)^{\frac{1}{t(1-t)}}.$$ Then
\begin{enumerate}
\item $f$ is decreasing on $(0,1)$ if $c<1$ and
\item $f$ is increasing on $(0,1)$ if $c>1.$
\end{enumerate}
\end{lemma}
\begin{proof}
Let $F(t)=\log f(t).$ Then
$$F(t)=\frac{t\log c+\log(1-t+t c^{-1})}{t(1-t)},$$
and
$$F'(t)=-\frac{g(c)}{(1-t)c+t},$$ where $$g(c)=t (-1 + c + t - c t) + (c (-1 + t) -
    t) \left[t^2 \log c + (-1 + 2 t) \log(1 + (-1 + c^{-1}) t)\right].$$
    Now
    $$g'(c)=\frac{t-1}{c}h(c),$$ where $$h(c)=(-1 + c) (-1 + t) t + c\;t^2 \log c +
 c (-1 + 2 t) \log(1 + (-1 + c^{-1}) t).$$ Moreover,
 $$h'(c)=t^2\log c+(2t-1)\left[\frac{(1-c)(t-1)t}{(1-t)c+t}+\log(1-t+t c^{-1})\right]$$ and
 $$h''(c)=\frac{t^2\;k(c)}{c((1-t)c+t)^2},\;k(c)=((1-t)c+t)^2+1-2t.$$
 Now
 $k'(c)=2(1-t)((1-t)c+t)>0$ and $k$ is increasing in $c>0.$ The following table summarizes the conclusion.
\begin{center}
\begin{tabular}{|c|c|}
  \hline
  $c<1$ & $c>1$ \\
  \hline
  $k(c)\geq k(0)=(t-1)^2\geq 0$.  &   $k(c)\geq k(1)=2-2t\geq 0$\\
 $\Rightarrow h''(c)\geq 0 \Rightarrow h'(c)\leq h'(1)=0$ &$\Rightarrow h''(c)\geq 0\Rightarrow h'(c)\geq h'(1)=0$    \\
 $\Rightarrow h(c)\geq h(1)=0\Rightarrow g'(c)\leq 0$  &  $\Rightarrow h(c)\geq h(1)=0\Rightarrow g'(c)\leq 0$   \\
 $\Rightarrow g(c)\geq g(1)=0\Rightarrow F'(t)<0$ &  $\Rightarrow g(c)\leq g(1)=0\Rightarrow F'(t)\geq 0$\\
 $\Rightarrow f$ is decreasing.& $\Rightarrow f$ is increasing.\\
  \hline
\end{tabular}
\end{center}
 This completes the proof.
\end{proof}
\begin{corollary}\label{geo_har_cor_num}
Let $a,b>0$ and $0<\nu,\tau<1.$ If $(b-a)(\tau-\nu)\geq 0,$ then
$$a\#_{\tau}b\geq a!_{\tau}b\left(\frac{a\#_{\nu}b}{a!_{\nu}b}\right)^{\frac{\tau(1-\tau)}{\nu(1-\nu)}}.$$ On the other hand, if
$(b-a)(\tau-\nu)\leq 0,$ then
$$a\#_{\tau}b\leq a!_{\tau}b\left(\frac{a\#_{\nu}b}{a!_{\nu}b}\right)^{\frac{\tau(1-\tau)}{\nu(1-\nu)}}.$$
\end{corollary}

\section{Applications in $\mathbb{M}_n$}
Our matrix results fall into two sections. The first section presents results that we obtain from the above numerical results. Then we present the other type, which is independent from the above numerical results.
\subsection{Quadratic results}
\subsubsection{Heinz-type inequalities}
Following our matrix notations from the introduction, we define the weighted harmonic and arithmetic means for $A,B\in\mathbb{M}_n^{++}$ as follows
$$A!_t B=((1-t)A^{-1}+t B^{-1})^{-1}\;{\text{and}}\;A\nabla_t B=(1-t)A+t B, 0\leq t\leq 1,$$ with the convention that $A!B=A!_{\frac{1}{2}}B$ and $A\nabla B=A\nabla_{\frac{1}{2}}B.$ Moreover, we define the harmonic Heinz matrix means by
$$!_t(A,B)=\frac{A!_tB+A!_{1-t}B}{2}, A,B\in\mathbb{M}_n^{++}, 0\leq t\leq 1.$$

Among the strongest comparisons between Hermitian  matrices is the so called L\"{o}wener partial ordering $\leq$, where we write $A\leq B$ when $B-A\in\mathbb{M}_n^{+}$. Recall that this partial ordering is preserved under conjugation. That is, if $A$ and $B$ are Hermitian such that $A\leq B$, then $CAC^*\leq CBC^*$ for any $C\in\mathbb{M}_n$.

As mentioned earlier,  the Heinz-type inequality
\begin{equation}\label{heinz_type_har_applications}
 a!b\leq !_t(a,b)\leq \frac{a+b}{2}, a,b>0, 0\leq t\leq 1
\end{equation}
 can be easily proved. This entails the following matrix version, in which the notation $D(\lambda_j)$ will mean the diagonal matrix whose diagonal entries are $\{\lambda_j\}.$ The proof is based on a standard functional calculus argument, that we present for completeness. Moreover, all forthcoming results about L\"{o}wener partial ordering results are proved similarly. Hence, we present these results without proofs.
\begin{proposition}\label{first_matrix}
 Let $A,B\in\mathbb{M}_n^{++}$ and let $0\leq t\leq 1.$ Then
$$A!B\leq !_t(A,B)\leq A\nabla B.$$
\end{proposition}
\begin{proof}
We present the proof of the second inequality. Let $X=A^{-\frac{1}{2}}BA^{-\frac{1}{2}}.$ Then $X\in\mathbb{M}_n^{++}$. If we denote the eigenvalues of $X$ by $\{\lambda_j\}$, we have $\lambda_j>0$, since $A\in\mathbb{M}_n^{++}$. Since $\lambda_j>0,$ \eqref{heinz_type_har_applications} implies
$!_t(1,\lambda_j)\leq 1\nabla \lambda_j$ for each $1\leq j\leq n.$ But then
\begin{equation}\label{needed_1_matrix}
 D(!_t(1,\lambda_j))\leq D(1\nabla \lambda_j).
\end{equation}
 Now since $X\in\mathbb{M}_n^{++}$, it follows that $X=UD(\lambda_j)U^*$ for some unitary matrix $U.$ Conjugating \eqref{needed_1_matrix} with $U$ implies $!_t(I,X)\leq I\nabla X,$ where $I$ is the identity matrix in $\mathbb{M}_n$. Then conjugating this last inequality with $A^{\frac{1}{2}}$ implies the desired inequality.
\end{proof}

A quadratic refinement of the above inequality maybe obtained using the same argument  of Proposition \ref{first_matrix} applied to Corollary \ref{full_comp_heinz} as follows.
\begin{theorem}
 Let $A,B\in\mathbb{M}_n^{++}$ and let $0\leq \tau, t\leq 1.$ Then
$$!_t(A,B)+\frac{t(1-t)}{\tau(1-\tau)}\left[A\nabla B-!_{\tau}(A,B)\right]\leq A\nabla B$$ when $t\leq r(\tau)$ or $t\geq R(\tau).$ On the other hand, if $r(\tau)\leq t\leq R(\tau),$ the inequality is reversed.
\end{theorem}

In particular, when $\tau=\frac{1}{2},$ we obtain the following simpler form.
\begin{corollary}
 Let $A,B\in\mathbb{M}_n^{++}$ and let $0\leq t\leq 1.$ Then
$$!_t(A,B)+4t(1-t)\left[A\nabla B-A!B\right]\leq A\nabla B.$$
\end{corollary}
Determinants inequalities can be obtained as well, recalling two facts
\begin{itemize}
 \item If $A\in\mathbb{M}_n^{+}$ has eigenvalues $\{\lambda_i(A)\}$, then
\begin{equation}\label{det_eigen_needed}
 \det A=\prod_{i=1}^{n}\lambda_i(A).
\end{equation}
\item Minkowski inequality which states that when $\{a_i\}$ and $\{b_i\}$ are two sets of positive numbers, we have
\begin{equation}\label{minkowski}
\left(\prod_{i=1}^{n}a_i\right)^{\frac{1}{n}}+\left(\prod_{i=1}^{n}b_i\right)^{\frac{1}{n}}\leq \left(\prod_{i=1}^{n}(a_i+b_i)\right)^{\frac{1}{n}}.
\end{equation}
\end{itemize}

\begin{theorem}\label{det_thm}
 Let $A,B\in\mathbb{M}_n^{++}$ and let $0\leq \tau,t\leq 1.$ If $t\leq r(\tau)$ or $t\geq R(\tau),$ then
$$\det(A\nabla B)^{\frac{1}{n}}\geq \det(!_t(A,B))^{\frac{1}{n}}+\frac{t(1-t)}{\tau(1-\tau)}\det(A\nabla B-!_{\tau}(A,B))^{\frac{1}{n}}.$$
\end{theorem}
\begin{proof}
 For $A,B\in\mathbb{M}_n^{++}$, let $X=A^{-\frac{1}{2}}BA^{-\frac{1}{2}}.$ Now, using Corollary \ref{full_comp_heinz} then \eqref{minkowski}, we obtain
\begin{eqnarray*}
 \det(I\nabla X)^{\frac{1}{n}}&=&\prod_{i=1}^{n}\lambda_i(I\nabla X)^{\frac{1}{n}}\\
&=&\prod_{i=1}^{n}(1\nabla \lambda_i(X))^{\frac{1}{n}}\\
&\geq&\left\{\prod_{i=1}^{n}\left(!_t(1,\lambda_i(X))+\frac{t(1-t)}{\tau(1-\tau)}
(1\nabla\lambda_i(X)-!_{\tau}(1,\lambda_i(X)))\right)\right\}^{\frac{1}{n}}\\
&\geq&\left(\prod_{i=1}^{n}!_t(1,\lambda_i(X))\right)^{\frac{1}{n}}+\frac{t(1-t)}{\tau(1-\tau)}
\left(\prod_{i=1}^{n}(1\nabla\lambda_i(X)-!_{\tau}(1,\lambda_i(X)))\right)^{\frac{1}{n}}\\
&=&\left(\prod_{i=1}^{n}\lambda_i(!_t(I,X))\right)^{\frac{1}{n}}+\frac{t(1-t)}{\tau(1-\tau)}
\left(\prod_{i=1}^{n}\lambda_i(I\nabla X-!_{\tau}(I, X))\right)^{\frac{1}{n}}\\
&=&\det(!_t(I,X))^{\frac{1}{n}}+\frac{t(1-t)}{\tau(1-\tau)}\det\left(I\nabla X-!_{\tau}(I, X)\right)^{\frac{1}{n}}.
\end{eqnarray*}
Now multiplying both sides of the last inequality with $\det A$ and using basic properties of the determinant imply the desired inequality.
\end{proof}

\begin{proposition}
 Let $A,B\in\mathbb{M}_n$ have positive traces, and let $0\leq \tau,t\leq 1$. Then
$$!_t(\tr A, \tr B)+\frac{t(1-t)}{\tau(1-\tau)}(\tr(A\nabla B)-!_{\tau}(\tr A,\tr B))\leq \tr(A\nabla B)$$ when $t\leq r(\tau)$ or $t\geq R(\tau).$ The inequality is reversed if $r(\tau)\leq t\leq R(\tau).$
\end{proposition}
\begin{proof}
 The result follows immediately from Corollary \ref{full_comp_heinz}, on letting $a=\tr A, b=\tr B$ and noting that the trace functional is additive.
\end{proof}

On the other hand, Corollary \ref{cor_1/2_multi} can be used to prove some multiplicative matrix versions as follows. In the following computations, we have used the fact that when $A,B\in\mathbb{M}_n^{++}$ commute, then powers of $A$ and $B$ also commute.
\begin{theorem}
 Let $A,B\in\mathbb{M}_n^{++}$ be commuting and let $0\leq t\leq 1.$ Then
$$\left(\frac{AB^{-1}+BA^{-1}+2I}{4}\right)^{4t(1-t)}!_t(A,B)\leq A\nabla B.$$
\end{theorem}
\begin{proof}
 Simplifying the inequality of Corollary \ref{cor_1/2_multi}, we obtain
$$\left(\frac{ab^{-1}+ba^{-1}+2}{4}\right)^{4t(1-t)}!_t(a,b)\leq a\nabla b.$$ Letting $a=1, X=A^{-\frac{1}{2}}BA^{-\frac{1}{2}}$ and applying the standard functional calculus argument as above, we obtain
\begin{equation}\label{needed_1_multi_matr}
\left(\frac{X^{-1}+X+2I}{4}\right)^{4t(1-t)}!_t(I,X)\leq I\nabla X.
\end{equation}
We simplify the terms appearing in the above inequality, as follows
\begin{itemize}
\item
 \begin{eqnarray}
\nonumber \left(\frac{X^{-1}+X+2I}{4}\right)^{4t(1-t)}&=&\left(\frac{A^{\frac{1}{2}}BA^{\frac{1}{2}}+A^{\frac{-1}{2}}BA^{\frac{-1}{2}}+2I}{4}\right)^{4t(1-t)}\\
\nonumber &=&\left[A^{\frac{-1}{2}}\left(\frac{AB^{-1}+BA^{-1}+2I}{4}\right)A^{\frac{1}{2}}\right]^{4t(1-t)}\\
\label{needed_2_multi_matr}&=&A^{\frac{-1}{2}}\left(\frac{AB^{-1}+BA^{-1}+2I}{4}\right)^{4t(1-t)}A^{\frac{1}{2}}.
\end{eqnarray}
\item
\begin{eqnarray}
 \nonumber !_t(I,X)&=&\frac{((1-t)I+tX^{-1})^{-1}+(tI+(1-t)X^{-1})^{-1}}{2}\\
\nonumber &=&\frac{((1-t)I+tA^{\frac{1}{2}}B^{-1}A^{\frac{1}{2}})^{-1}+(tI+(1-t)A^{\frac{1}{2}}B^{-1}A^{\frac{1}{2}})^{-1}}{2}\\
\nonumber &=& A^{\frac{-1}{2}} \frac{((1-t)A^{-1}+tB^{-1})^{-1}+(tA^{-1}+(1-t)B^{-1})^{-1}}{2}A^{\frac{-1}{2}}\\
\label{needed_3_multi_matr}&=&A^{\frac{-1}{2}} !_t(A,B)A^{\frac{-1}{2}}.
\end{eqnarray}
\end{itemize}
Now using \eqref{needed_2_multi_matr} and \eqref{needed_3_multi_matr},  \eqref{needed_1_multi_matr} becomes
$$A^{\frac{-1}{2}}\left(\frac{AB^{-1}+BA^{-1}+2I}{4}\right)^{4t(1-t)}!_t(A,B)A^{\frac{-1}{2}}\leq I\nabla X.$$ Conjugating both sides with $A^{\frac{1}{2}}$ implies the result.
\end{proof}

\subsubsection{The arithmetic-harmonic mean inequality}
The proof of the following proposition follows the same steps as that of Proposition \ref{first_matrix}, using the numerical versions in Proposition
\ref{arith_har_num_first}.
\begin{proposition}
Let $A,B\in\mathbb{M}_n^{++}$ and $0<\nu,\tau<1$ be such that $(\tau-\nu)(B-A)\geq 0.$ Then
$$\tau(1-\tau)(A\nabla_{\nu}B-A!_{\nu}B)\leq \nu(1-\nu)(A\nabla_{\tau}B-A!_{\tau}B).$$ If $(\tau-\nu)(B-A)\leq 0$ then the inequality is reversed.
\end{proposition}

On the other hand, using Proposition \ref{arith_har_num_first} and Corollary \ref{arith_har_square},  determinant versions maybe obtained in a similar way to Theorem \ref{det_thm}.
\begin{proposition}
Let $A,B\in\mathbb{M}_n^{++}$ and $0<\nu,\tau<1$ be such that $(\tau-\nu)(B-A)\geq 0.$ Then
$$\det(A!_{\tau}B)^{\frac{1}{n}}+\frac{\tau(1-\tau)}{\nu(1-\nu)}\det(A\nabla_{\nu}B-A!_{\nu}B)^{\frac{1}{n}}\leq \det(A\nabla_{\tau}B)^{\frac{1}{n}}$$
and
$$\det(A!_{\tau}B)^{\frac{2}{n}}+\frac{\tau(1-\tau)}{\nu(1-\nu)}\det(A\nabla_{\nu}B-A!_{\nu}B)^{\frac{2}{n}}\leq \det(A\nabla_{\tau}B)^{\frac{2}{n}}.$$
\end{proposition}

Further, Corollary \ref{arith_har_kanto_cor} implies the following interesting refinement of the arithmetic-harmonic mean inequality for matrices.
\begin{theorem}
Let $A,B$ be commuting matrices in $\mathbb{M}_n^{++}$ and let $0\leq t\leq 1.$ Then
$$\left(\frac{BA^{-1}+AB^{-1}+2I_n}{4}\right)^{4t(1-t)}(A!_{t}B)\leq A\nabla_{t}B.$$
\end{theorem}

\subsubsection{The geometric-harmonic mean inequality}
Following the same logic of Theorem \ref{first_matrix} with the aid of Corollary \ref{geo_har_add}, we obtain the following matrix version.
\begin{proposition}
Let $A,B\in\mathbb{M}_n^{++}$ and $0<\nu,\tau<1.$ If $(\tau-\nu)(B-A)\geq 0,$ then
$$\tau(1-\tau)(A\#_{\nu}B-A!_{\nu}B)\leq \nu(1-\nu)(A\#_{\tau}B-A!_{\tau}B),$$
where $A\#_tB=A^{\frac{1}{2}}\left(A^{-\frac{1}{2}}BA^{-\frac{1}{2}}\right)^{t}A^{\frac{1}{2}}$ is the matrix geometric mean.
\end{proposition}
Again, applying the logic of Theorem \ref{det_thm},  with the aid of Corollary \ref{geo_har_add}, we obtain the determinant version.
\begin{proposition}
Let $A,B\in\mathbb{M}_n^{++}$ and $0<\nu,\tau<1.$ If $(\tau-\nu)(B-A)\geq 0,$ then
$$\det(A!_{\tau}B)^{\frac{1}{n}}+\frac{\tau(1-\tau)}{\nu(1-\nu)}\det(A\#_{\nu}B-A!_{\nu}B)^{\frac{1}{n}}\leq \det(A\#_{\tau}B)^{\frac{1}{n}}.$$
\end{proposition}

\subsection{Young-type inequality}
As mentioned in the introduction, the inequality
$$a\#_t b\leq a\nabla_t b, a,b>0, 0\leq t\leq 1$$ is well known. This inequality is usually referred to as Young's inequality.\\
Matrix versions of this inequality have different forms. For example, applying the functional calculus argument of the above section implies the well known matrix version
$$A\#_tB\leq A\nabla_t B, A,B\in\mathbb{M}_n^{++}, 0\leq t\leq 1.$$ The other matrix version of Young's inequality is a unitarily invariant norm version stating, for $0\leq t\leq 1,$
\begin{equation}\label{young_matr_norm}
 \||A^{1-t}XB^{t}\||\leq (1-t)\||AX\||+t\||XB\||, A,B\in\mathbb{M}_n^{++}, X\in\mathbb{M}_n
\end{equation}
 for any unitarily invariant norm $\||\;\;\||,$ \cite{omarkittaneh}. Recall that these are norms satisfying $\||UXV\||=\||X\||$ for any $X\in\mathbb{M}_n^{+}$ and any unitary matrices $U,V.$ The proof of \eqref{young_matr_norm} is based on two inequalities; the first is the H\"{o}lder matrix inequality \cite{k}, stating for $A,B\in\mathbb{M}_n^{+}, X\in\mathbb{M}_n$ and $0\leq t\leq 1,$
\begin{equation}\label{holder_matr}
\||A^{1-t}XB^{t}\||\leq \||AX\||^{1-t}\||XB\||^{t},
\end{equation}
then applying the Young inequality.\\
However, a stronger version of \eqref{young_matr_norm} can be shown for the Hilbert-Shmidt norm as follows \cite{Bh,kosaki},
\begin{equation}\label{young_matr_2_norm}
\|A^{1-t}XB^{t}\|_2\leq \|(1-t)AX+t XB\|_2,
\end{equation}
where $\|\;\;\|_2$ is the Hilbert-Schmidt norm defined, for $A=[a_{ij}]\in\mathbb{M}_n$ by $$\|A\|_2=\left(\sum_{i,j}|a_{ij}|^2\right)^{\frac{1}{2}}.$$ It is well known that $\|\;\;\|_2$ is unitarily invariant.\\
A stronger version was shown in \cite{ando} for the singular values as follows
$$s_j(A^{1-t}B^{t})\leq s_j((1-t)A+t B),$$ where $s_j$ is the $j-$th singular value, when written in a decreasing order. This last inequality implies $\||A^{1-t}B^{t}\||\leq \||(1-t)A+t B\||$ for any unitarily invariant norm. However, this inequality is not valid when an arbitrary $X$ is included. That is, the inequality $\||A^{1-t}XB^{t}\||\leq \||(1-t)AX+t XB\||$ is not necessarily valid for norms other than the $\|\;\;\|_2$ norm, \cite{Bh}.

Our next result is a harmonic-mean version of \eqref{young_matr_2_norm}. We remark that such versions have never been seen in the literature. We claim that introducing such a result will attract researchers in the field and will open the door for more results.\\
Before stating our result, we remind the reader that a function $f:(0,\infty)\to\mathbb{R}$ is said to be convex if
$$f\left(\sum_{i=1}^{n}\alpha_i x_i\right)\leq \sum_{i=1}^{n}\alpha_i f(x_i),$$ for $\{x_i\}\subset (0,\infty)$ and $\{\alpha_i\}$ is a convex sequence. That is, $\alpha_i\geq 0$ and $\sum_{i}\alpha_i=1.$
\begin{theorem}\label{thm_har_young}
 Let $A,B\in\mathbb{M}_n^{++}, X\in\mathbb{M}_n$ and $0\leq t\leq 1.$ If $(1-t)A^{-1}X+t XB^{-1}\not=0$ and $(1-t)AX+t XB\not=0$, then
\begin{equation}\label{young_harmonic_matr}
\|X\|_2^2\;\|(1-t)A^{-1}X+t XB^{-1}\|_2^{-1}\leq \|(1-t)AX+t XB\|_2.
\end{equation}
\end{theorem}
\begin{proof}
If $X=0,$ the result is trivial. So, we assume that $X\not=0.$ Suppose first that $\|X\|_2=1.$ Since $\|\;\;\|_2$ is a unitarily invariant norm, we have $\|U^*XV\|_2=1$ for any unitary matrices $U^*,V.$ Therefore, if we let $Y=U^*XV$, then
\begin{equation}\label{needed_1_har_young}
\sum_{i,j}|y_{ij}|^2=\|Y\|_2^2=\|X\|_2^2=1.
\end{equation}
Consequently, if $f:(0,\infty)\to\mathbb{R}$ is a convex function, we have
\begin{equation}\label{needed_2_har_young}
f\left(\sum_{i,j}|y_{ij}|^2x_i\right)\leq \sum_{i,j}|y_{ij}|^2f(x_i),
\end{equation}
for any set $\{x_i\}\subset (0,\infty).$\\
Now since $A,B\in\mathbb{M}_n^{+}$, there are unitary matrices $U,V$ such that
$$A=U D(\lambda_i)U^*\;{\text{and}}\;B=V D(\mu_j)V^*.$$  Adopting these notations, it is trivial to see that
$$(1-t)AX+t XB=U\left([(1-t)\lambda_i+t\mu_j]\circ [y_{ij}]\right)V^*,$$ where $\circ$ stands for the Schur product of two matrices. That is, the entry wise multiplication.\\
Let $f(x)=x^{-1}, x>0$. Then $f$ is convex, and
\begin{eqnarray*}
\left(\|(1-t)AX+t XB\|_2^2\right)^{-1}&=&f\left(\|(1-t)AX+t XB\|_2^2\right)\\
&=&f\left(\|U\left([(1-t)\lambda_i+t\mu_j]\circ [y_{ij}]\right)V^*\|_2^2\right)\\
&=&f\left(\sum_{i,j}|y_{ij}|^2((1-t)\lambda_i+t\mu_j)^2\right)\\
&\leq&\sum_{i,j}|y_{ij}|^2f(((1-t)\lambda_i+t\mu_j)^2)\;({\text{by}}\;\eqref{needed_2_har_young})\\
&=&\sum_{i,j}|y_{ij}|^2f^2((1-t)\lambda_i+t\mu_j)\\
&\leq&\sum_{i,j}|y_{ij}|^2\left((1-t)f(\lambda_i)+t f(\mu_j)\right)^2\\
&=&\sum_{i,j}|y_{ij}|^2\left((1-t)\lambda_i^{-1}+t\mu_j^{-1}\right)^2\\
&=&\|(1-t)A^{-1}X+t XB^{-1}\|_2^2.
\end{eqnarray*}
This proves that
\begin{equation}\label{needed_3_har_young}
 \|(1-t)A^{-1}X+t XB^{-1}\|_2^{-1}\leq \|(1-t)AX+t XB\|_2,
\end{equation}
when $\|X\|_2=1.$ On the other hand, if $\|X\|_2\not=1,$ replace $X$ by $\frac{X}{\|X\|_2}$ in \eqref{needed_3_har_young} to get the desired inequality.
\end{proof}
In particular, when $X=I,$ the identity, we get the following.
\begin{corollary}\label{cor_har_young}
Let $A,B\in\mathbb{M}_n^{++}.$ Then,
$$n\|(1-t)A^{-1}+t B^{-1}\|_2^{-1}\leq \|(1-t)A+t B\|_2.$$
\end{corollary}
Observe that the above corollary is sharp in many cases. For example, when $n=1$ we have the equality attained. Moreover, when $A=B=I,$ both sides are equal to $\sqrt{n}.$

We emphasize that, for $0\leq t\leq 1,$ the inequality $A!_tB\leq A\nabla_t B$ follows immediately from the scalar inequality $a!_tb\leq a\nabla_tb.$ Then, the inequality $\||((1-t)A^{-1}+tB^{-1})^{-1}\||\leq \||(1-t)A+tB\||$ follows for any unitarily invariant norm $\||\;\;\||.$ Notice that Theorem \ref{thm_har_young} presents a variant of this inequality for the 2-norm including $X$. Furthermore, if $C\in\mathbb{M}_n^{++}$, one can easily show that $n\|C\|_2^{-1}\leq \|C^{-1}\|_2.$ Indeed, if $f(x)=x^{-1}$ and $C=UD(\lambda_i)U^*$ is the spectral decomposition of $C$, then
\begin{eqnarray*}
\|C\|_2^{-2}&=&f\left(\sum_{i}\lambda_i^2\right)\\
&=&f\left(\sum_{i}\frac{1}{n} n\lambda_i^2\right)\\
&\leq&\sum_{i}\frac{1}{n}f(n\lambda_i^2)\;({\text{by\;convexity\;of}}\;f)\\
&=&\frac{1}{n^2}\|C^{-1}\|_2^2.
\end{eqnarray*}
That is, $n\|C\|_2^{-1}\leq \|C^{-1}\|_2.$ Letting $C=(1-t)A^{-1}+tB^{-1},$ we obtain
$$n\|(1-t)A^{-1}+tB^{-1}\|_2^{-1}\leq \|((1-t)A^{-1}+tB^{-1})^{-1}\|_2.$$ Therefore, Corollary \ref{cor_har_young} follows from this observation.

A stronger version of Theorem \ref{thm_har_young} maybe obtained noting log-convexity of the function $f(x)=x^{-1}, x>0.$ The proof of the next result follows the same logic of Theorem \ref{thm_har_young}, noting that log-convexity of $f(x)=x^{-1}$ implies $f((1-t)\lambda_i+t \mu_j)\leq f^{1-t}(\lambda_i)f^{t}(\mu_j).$
\begin{theorem}
 Let $A,B\in\mathbb{M}_n^{++}, X\in\mathbb{M}_n$ and $0\leq t\leq 1.$ If $(1-t)A^{-1}X+t XB^{-1}\not=0$ and $A^{1-t}XB^{t}\not=0$, then
$$
\|X\|_2^2\;\|(1-t)A^{-1}X+t XB^{-1}\|_2^{-1}\leq \|A^{1-t}XB^{t}\|_2.
$$
\end{theorem}

\end{document}